%% file: resubmit_birat_April21.tex
\documentclass{article}
\input{packages.tex}

\usepackage[utf8]{inputenc}

\title{Birational sequences and the tropical Grassmannian}
\author{Lara Bossinger\footnote{Supported by "Programa de Becas Posdoctorales en la UNAM 2018" Instituto de Matem\'aticas, Universidad Nacional Aut\'onoma de M\'exico.}}
\date{}
\begin{document}

\maketitle
\vspace{-.75cm}
\begin{abstract}
   We introduce iterated sequences for Grassmannians, a new class of Fang-Fourier-Littelmann's birational sequences and explain how they give rise to points in the tropical Grassmannian.
    For $\Gr(2,\mathbb C^n)$ we show that the associated valuations induce toric degenerations. 
    We describe recursively the vertices of the corresponding Newton--Okounkov polytopes, which are particular vertices of a hypercube and hence integral.
    We show further that every toric degenerations of $\Gr(2,\mathbb C^{n})$ constructed using the tropical Grassmannian due to Speyer and Sturmfels can be recovered by iterated sequences.
\end{abstract}

\section{Introduction}
Toric degenerations of Grassmannians or more generally flag and spherical varieties have been a vivid research topic during the last 20 years.
A novel framework, called \emph{birational sequences}, stemming from representation theory was given by Fang, Fourier and Littelmann in \cite{FFL15}. 
Within this framework, we introduce and analyze new birational sequences for Grassmannians and discover surprising connections to the tropical Grassmannian by Speyer--Sturmfels \cite{SS04}.

To be more precise we need to introduce a bit of notation.
A \emph{toric degeneration} of a projective variety $X$ is a flat morphism $\pi:\mathcal X\to \mathbb A^d$ with generic fiber $\pi^{-1}(t)$ for $t\not=0$ isomorphic to $X$ and $\pi^{-1}(0)$ a projective toric variety.
Consider $SL_n$ over $\mathbb C$ with Borel subgroup $B$ of upper triangular matrices.
Let $P_k$ be the parabolic subgroup of block-upper triangular matrices with blocks of size $k\times k$ and $(n-k)\times (n-k)$. In particular, the Grassmannian $\Gr(k,\mathbb C^n)$ is isomorphic to $SL_n/P_k$. 
Denote its dimension by $d$.
Further, let $U^-$ be lower triangular matrices with $1$'s along the diagonal.
The root system of type $\mathtt A_{n-1}$ is  $R=\{\varepsilon_i-\varepsilon_j\}_{i\neq j}$ where $\{\varepsilon_i\}_i$ is the standard basis of $\mathbb R^n$. 
For every positive root $\beta\in R^+$ we have the one-parameter root subgroup $U_{-\beta}\subset U^-$.
In this setting the notion of birational sequences due to Fang, Fourier and Littelmann \cite{FFL15} applies:
a sequence $S=(\beta_1,\dots,\beta_d)$ of positive roots is called \emph{birational} for $\Gr(k,\mathbb C^n)$ if the multiplication map
\[
U_{-\beta_1}\times \dots \times U_{-\beta_d}\to U^-
\]
has image birational to $\Gr(k,\mathbb C^n)$. 
Birational sequences give rise to coordinates on the Grassmannian.
These can be used to define valuations on the homogeneous coordinate ring.
The construction allows to compute certain values of the valuations explicitly via representations of the Lie algebra $\mathfrak{sl}_n$.
If the valuation has finitely generated value semi-group it induces a toric degeneration of the Grassmannian by \cite{An13}.
Their setting unifies many known constructions of toric degenerations in representation theory, such as \cite{GL96,Cal02,AB04,FFL11}.
However, besides the previously known cases few \emph{new} birational sequences were discovered so far.
A central question is whether toric degeneration\textcolor{red}{s} constructed from tropical geometry (see e.g. \cite{SS04} for the case of Grassmannians) or cluster algebras (see e.g. \cite{RW17}) can be recovered using birational sequences.

\medskip
In this paper we study birational sequences that do not correspond to the previously known ones. One way to construct them is provided by the following central Lemma (see also Lemma~\ref{lem:birat}). It allows to obtain a birational sequences for $\Gr(k,\mathbb C^{n+1})$ from a birational sequence for $\Gr(k,\mathbb C^n)$.

\begin{lemma*}
Let $S=(\beta_1,\dots,\beta_d)$ be a birational sequence for $\Gr(k,\mathbb C^n)$ and choose $i_1,\dots,i_k$ pairwise differently from $\{1,\dots,n\}$. Then the following is a birational sequence for $\Gr(k,\mathbb C^{n+1})$:
\[
S'=(\varepsilon_{i_1}-\varepsilon_{n+1},\dots,\varepsilon_{i_k}-\varepsilon_{n+1},\beta_1,\dots,\beta_d).
\]
\end{lemma*}

Inspired by the Lemma we define a new class of birational sequences:
a birational sequence $S$ for $\Gr(k,\mathbb C^n)$ constructed from a birational sequence for $\Gr(k,\mathbb C^{k+1})$ by applying the Lemma $n-k-1$ many times is called \emph{iterated} (see Definition~\ref{def:itseq}).

\medskip 

Focusing on iterated sequences for $\Gr(2,\mathbb C^n)$, we show further that the corresponding valuations have the necessary property to induce toric degenerations (relying on a result of \cite{B-quasival}).
Fixing the Pl\"ucker embedding of $\Gr(2,\mathbb C^n)$ the corresponding Newton--Okounkov polytopes are in fact integral polytopes inside the hypercube\footnote{We consider the hypercube as the convex hull of all elements in $\{0,1\}^{2(n-2)}\subset \mathbb R_{\ge 0}^{2(n-2)}$.} in $\mathbb R_{\ge 0}^{2(n-2)}$ (see Corollary~\ref{cor:hypersimplex}).
The vertices are given by the images of the valuation on Pl\"ucker coordinates, which in this case form a Khovanskii basis.

Moreover, we identify iterated sequences with maximal cones in $\trop(\Gr(k,\mathbb C^n))$, the tropical Grassmannian \cite{SS04}.
For $k=2$ we give the explicit Algorithm~\ref{alg:tree assoc to seq} to do so. 
For an iterated sequence $S$ denote by $C_S\subset\trop(\Gr(2,\mathbb C^n))$ the output of the algorithm.
We say two toric degenerations of a projective variety $X$ are \emph{equivalent} if their special fibers are isomorphic toric varieties.
Our main results (Theorems~\ref{thm: seq and trop k}\&\ref{thm:main birat}) can be summarized as follows:

\begin{theorem*}
Let $S$ be an iterated sequence for $\Gr(k,\mathbb C^n)$. Then there exists a maximal cone $C\subset \trop(\Gr(k,\mathbb C^n))$ such that the toric degeneration of $\Gr(k,\mathbb C^n)$ induced by $S$ is equivalent to the one given by the maximal cone $C\subset \trop(\Gr(k,\mathbb C^n))$.

Moreover, for $k=2$ the cone $C$ equals the cone $C_S$ which is the output Algorithm~\ref{alg:tree assoc to seq}. Conversely, for every maximal
cone $C\subset \trop(\Gr(2,\mathbb{C}^{n}))$ there exists an iterated
sequence $S$ such that the induced toric degenerations of $C$ and $S$ are equivalent.
\end{theorem*}

\begin{remark}
Combining with the results of \cite{BFFHL}, where toric degenerations of $\Gr(2,\mathbb C^n)$ from plabic graphs are studied, our main result implies that iterated sequences not only recover all possible toric degenerations of $\Gr(2,\mathbb C^n)$ constructed using the tropical Grassmannian, but also all those constructed using plabic graphs (as in \cite{RW17}).
Said differently, up to equivalence of toric degenerations the following sets are in one-to-one correspondence:
\[
\left\{\begin{matrix}
\text{toric degenerations}\\
\text{ of } \Gr(2,\mathbb C^n) \text{ from }\\
\text{iterated sequences}
\end{matrix}\right\} \overset{\text{Theorems~\ref{thm: seq and trop k}\&\ref{thm:main birat}}}{\longleftrightarrow}
\left\{\begin{matrix}
\text{toric degenerations}\\
\text{ of } \Gr(2,\mathbb C^n) \text{ from }\\
 \trop(\Gr(2,\mathbb C^n))
\end{matrix}\right\} \overset{\text{\cite{BFFHL}}}{\longleftrightarrow}
\left\{\begin{matrix}
\text{toric degenerations}\\
\text{ of } \Gr(2,\mathbb C^n) \text{ from }\\
\text{plabic graphs}
\end{matrix}\right\}
\]

\end{remark}

\begin{figure}[h]
\begin{center}
\begin{tikzpicture}[scale=.14]

\draw (0,0) -- (0,-2.5);
\draw (0,0) -- (2.5,2);
\draw (0,0) -- (-2.5,2);

\draw[thick,->] (0,-3.5) -- (0,-6);

\begin{scope}[yshift=-8cm] 
\draw (-3,2)-- (-1.5,0) -- (1.5,0) -- (3,2);
\draw (-3,-2) -- (-1.5,0);
\draw (1.5,0) -- (3,-2);

\draw[thick,->] (0,-3.5) -- (0,-6);
\end{scope}

\begin{scope}[yshift=-16cm,xshift=-1.5cm] 
\draw (-3,2)-- (-1.5,0) -- (4.5,0) -- (6,2);
\draw (-3,-2) -- (-1.5,0);
\draw (4.5,0) -- (6,-2);
\draw (1.5,0) -- (1.5,-2);

\draw[thick,->] (7,-3.5) -- (8.25,-6);
\draw[thick,->] (-4.25,-3.5) -- (-5.75,-6);
\end{scope}

\begin{scope}[yshift=-24cm,xshift=-10cm] 
\draw (-3,2)-- (-1.5,0) -- (6.5,0) -- (8,2);
\draw (-3,-2) -- (-1.5,0);
\draw (6.5,0) -- (8,-2);
\draw (1,0) -- (1,-2);
\draw (4,0) -- (4,-2);
\draw[thick,->] (1.5,-3) -- (-.5,-6);
\draw[thick, ->] (9.5,-3) -- (12,-6);
\begin{scope}[xshift=16cm] 
\draw (-3,2)-- (-1.5,0) -- (4.5,0) -- (6,2);
\draw (-3,-2) -- (-1.5,0);
\draw (4.5,0) -- (6,-2);
\draw (1.5,0) -- (1.5,-2) -- (0,-4);
\draw (1.5,-2) -- (3,-4);
\draw[thick,->] (4.5,-3.5) -- (6,-6);
\end{scope}
\end{scope}

\begin{scope}[yshift=-33cm,xshift=-13cm]
\draw (-3,2)-- (-1.5,0) -- (6.5,0) -- (8,2);
\draw (-3,-2) -- (-1.5,0);
\draw (6.5,0) -- (8,-2);
\draw (.5,0) -- (.5,-2);
\draw (2.5,0) -- (2.5,-2);
\draw (4.5,0) -- (4.5,-2);
\draw[thick,->] (-4.5,-3) -- (-7,-6);
\draw[thick,->] (2.5,-3) -- (3.5,-6);
\draw[thick, ->] (9.5,-2.5) -- (17,-6);
\begin{scope}[xshift=20cm] 
\draw (-3,2)-- (-1.5,0) -- (6.5,0) -- (8,2);
\draw (-3,-2) -- (-1.5,0);
\draw (6.5,0) -- (8,-2);
\draw (1,0) -- (1,-2);
\draw (4,0) -- (4,-2);
\draw (4,-2) -- (2.5,-4);
\draw (4,-2) -- (5.5,-4);
\draw[thick,->] (.75,-3.5) -- (-.5,-6.5);
\draw[thick,->] (-3.5,-2.5) -- (-10.5,-6);
\draw[thick, ->] (9.5,-3) -- (12,-6);
\end{scope}
\end{scope}

\begin{scope}[yshift=-42cm,xshift=-23cm]
\draw (-3,2)-- (-1.5,0) -- (6.5,0) -- (8,2);
\draw (-3,-2) -- (-1.5,0);
\draw (6.5,0) -- (8,-2);
\draw (0.1,0) -- (0.1,-2);
\draw (1.7,0) -- (1.7,-2);
\draw (3.2,0) -- (3.2,-2);
\draw (4.9,0) -- (4.9,-2);
\node at (2.5,-4) {$\vdots$};
\begin{scope}[xshift=14cm] 
\draw (-3,2)-- (-1.5,0) -- (6.5,0) -- (8,2);
\draw (-3,-2) -- (-1.5,0);
\draw (6.5,0) -- (8,-2);
\draw (.5,0) -- (.5,-2);
\draw (2.5,0) -- (2.5,-2);
\draw (4.5,0) -- (4.5,-2);
\draw (4.5,-2) -- (3,-4);
\draw (4.5,-2) -- (6,-4);
\node at (1.5,-4) {$\vdots$};
\end{scope}

\begin{scope}[xshift=27cm]
\draw (-3,2)-- (-1.5,0) -- (6.5,0) -- (8,2);
\draw (-3,-2) -- (-1.5,0);
\draw (6.5,0) -- (8,-2);
\draw (.5,0) -- (.5,-2);
\draw (2.5,0) -- (2.5,-2);
\draw (2.5,-2) -- (1,-4);
\draw (2.5,-2) -- (4,-4);
\draw (4.5,0) -- (4.5,-2);
\node at (2.5,-4.5) {$\vdots$};
\begin{scope}[xshift=14cm] 
\draw (-3,2)-- (-1.5,0) -- (6.5,0) -- (8,2);
\draw (-3,-2) -- (-1.5,0);
\draw (6.5,0) -- (8,-2);
\draw (.5,0) -- (.5,-2);
\draw (.5,-2) -- (-1,-4);
\draw (.5,-2) -- (2,-4);
\draw (4.5,0) -- (4.5,-2);
\draw (4.5,-2) -- (3,-4);
\draw (4.5,-2) -- (6,-4);
\node at (2.5,-4.5) {$\vdots$};
\end{scope}
\end{scope}
\end{scope}

\end{tikzpicture}
\end{center}
\caption{The tree graph $\mathcal T$ from level ($\#$of leaves) 3 to 8.}\label{fig:treegraph}
\end{figure}

We analyze the combinatorial structure of iterated sequences in terms of trivalent trees with $n$ leaves that are in correspondence with maximal cones of $\trop(\Gr(2,\mathbb C^n))$.
Let $\mathcal T$ be the infinite graph with vertices corresponding to (unlabelled) trivalent trees.
Two vertices $\mathtt T_1,\mathtt T_2$ of $\mathcal T$ are connected by an directed edge $\mathtt T_1\to \mathtt T_2$, if $\mathtt T_2$ can be obtained from $\mathtt T_1$ by adding a leaf-edge (see Figure~\ref{fig:treegraph}).
The unique source of $\mathcal T$ is the trivalent tree with three leaves.
Algorithm~\ref{alg:tree assoc to seq} yields the following corollary:

\begin{corollary*}
Every iterated sequence $S$ for $\Gr(2,\mathbb C^n)$ yields a unique path from the source of $\mathcal T$ to the trivalent tree with $n$ leaves corresponding to the equivalence class\footnote{Equivalence classes are considered with respect to the action of the symmetric group $S_n$ on the Pl\"ucker coordinates, see \S\ref{sec:tropical} below Remark~\ref{rem:families}} of the maximal cone $C_S\subset\trop(\Gr(2,\mathbb C^n))$.
\end{corollary*}

The paper is structured as follows: after setting up the notation in \S\ref{sec:notation} we recall necessary notions on valuations in \S\ref{sec:pre valuations}. 
Reminders on tropical geometry can be found in \S\ref{sec:tropical}.
In \S\ref{sec:pre_birat} we introduce iterated (birational) sequences for Grassmannians and recall the construction of the associated valuations due to \cite{FFL15}.
In \S\ref{sec:proof} we focus on iterated sequences for $\Gr(2,\mathbb C^n)$ and prove the main result.

\medskip

{\bf Acknowledgements:} The result of this paper is part of my PhD thesis \cite{Thesis} supervised by Peter Littelmann at the University of Cologne. 
I am grateful for his advice and support throughout my PhD. 
Further, I would like to thank  Xin Fang, Ghislain Fourier and Alfredo N\'ajera Ch\'avez for inspiring and helpful discussions and an anonymous referee for asking a question that lead to Corollary~\ref{cor:seq of cones}.

\section{Notation}\label{sec:notation}
Let $I_{k,n}$ be the set of ordered integer sequences $\ui=(i_1<\dots<i_k)$ with $1\le i_1$ and $i_k\le n$.
We consider the Grassmannian $\Gr(k,\mathbb C^n)$ of $k$-dimensional subspaces in $\mathbb C^n$ with its Pl\"ucker embedding.
It is the vanishing of the \emph{Pl\"ucker ideal} $\mathcal I_{k,n}\subset \mathbb C[p_{\uj}\mid \uj\in I_{k,n}]$ in the polynomial ring on \emph{Pl\"ucker variables} $p_{\uj}$.
To describe $\mathcal I_{k,n}$ more explicitly we need some more notation.

Denote $\{1,\dots,n\}$ by $[n]$. Consider $\uj=(j_1<\dots<j_{k+1})\in I_{k+1,n}$ and $j_s\in \uj$, i.e. $1\le s\le k+1$.
Then $\uj\setminus j_s$ denotes the sequence in $I_{k,n}$ obtained by removing $j_s$ from $\uj$.
For $\ui=(i_1<\dots<i_{k-1})\in I_{k-1,n}$ and $j\in [n]$ define 
\[
\ui\cup j: = (i_1<\dots<i_r <j<i_{r+1}<\dots<i_{k-1}).
\]
Set $\ell (\ui,j):= k-(r+1)$.
Note that if $j=i_l$ for some $1\le l\le k-1$ then $\ui \cup j\not \in I_{k,n}$, in this case we set $p_{\ui\cup j}=0$.
The Pl\"ucker ideal $\mathcal I_{k,n}$ is generated by elements of form
\begin{eqnarray}\label{eq:plucker rel}
R_{\ui,\uj}:=\sum_{s=1}^{k+1} (-1)^{\ell(\ui,j_s)} p_{\ui\cup j_s}p_{\uj\setminus j_s},
\end{eqnarray}
for $\ui\in I_{k-1,n}$ and $\uj\in I_{k+1,n}$ (see e.g. \cite[\S4.2.2]{LB15}). 
For $k=2$ we simplify notation by $p_{ij}:=p_{(i<j)}$ for $(i<j)\in I_{2,n}$.
Then the Pl\"ucker relations are of form:
\[
R_{r,s,u,v}:=R_{(r),(s<u<v)}=p_{rs}p_{uv}-p_{ru}p_{sv}+p_{rv}p_{su} \in\mathbb C[p_{ij}]_{i<j},
\]
with $1\le r<s<u<v\le n$.
Let $A_{k,n}:=\mathbb C[p_{\uj}]_{\uj\in I_{k,n}}/\mathcal I_{k,n}\cong \mathbb C[\Gr(k,\mathbb C^k)]$, which is the homogeneous coordinate ring of the Grassmannian.
The cosets $\bar p_{\uj}\in A_{k,n}$ are called \emph{Pl\"ucker coordinates}.

We recall some standard facts about Lie algebras, algebraic groups and their representations that will be used throughout the rest of the paper.
Let $\lie{sl}_n$ be the Lie algebra corresponding to $SL_n$ and fix a Cartan decomposition $\lie{sl}_n=\lie n^-\oplus \lie h\oplus \lie n^+$.
The Cartan subalgebra $\lie h$ consists of diagonal traceless matrices, and $\lie n ^-$ (resp. $\lie n^+$) of strictly lower (resp. upper) triangular matrices.

In the type $\mathtt A_{n-1}$ root system $R=\{\varepsilon_i-\varepsilon_j\}_{i\neq j\in[n]}\subset \mathbb R^n$ we have positive roots $R^+=\{\varepsilon_i-\varepsilon_j\}_{i<j}$ and simple roots $\{\varepsilon_i-\varepsilon_{i+1}\}_{i=1,\dots,n-1}$.
Further, set $R_k^+:=\{\varepsilon_i-\varepsilon_j \mid 1\le i\le k\le j\le n\}$.
For every positive root $\varepsilon_i-\varepsilon_j$ we have a root operator $f_{i,j}\in\lie n^-$: 
it is the elementary matrix with only non-zero entry being $1$ in the $(j,i)$-position.

Let $\Lambda$ denote the weight lattice, generated by fundamental weights $\omega_1,\dots,\omega_{n-1}$.
The $k$th fundamental representation of $\lie{sl}_n$ is $V(\omega_k)=\bigwedge^k\mathbb C^n$.
It is cyclically generated by a highest weight vector $v_{\omega_k}$ over the universal enveloping algebra $U(\mathfrak n^-)$.

\begin{example}\label{exp: fund reps 1 and 2}
For $V(\omega_2)$ fix the basis $\{e_i \wedge e_j \mid 1\le i<j\le n\}$. Then the action of $\lie n^-$ is given by
\[
f_{i,j}\cdot (e_k\wedge e_l)=f_{i,j}\cdot e_k\wedge e_l + e_k\wedge f_{i,j}\cdot e_l=\left\{\begin{matrix} 
e_{j+1}\wedge e_l, & \text{ if } k=i,\\
e_k\wedge e_{j+1}, & \text{ if } l=i,\\
0, & \text{ otherwise.} 
\end{matrix}\right.
\]
We choose $e_1\wedge e_2$ as the highest weight vector $v_{\omega_2}$.
\end{example}

Identifying the Grassmannian with the quotient $SL_n/P_k$, the Pl\"ucker embedding can be written as $\Gr(k,\mathbb C^n)\hookrightarrow \mathbb P(V(\omega_k))$.
For $r\ge 1$ let $V(r\omega_k)^*$ denote the vector space dual to $V(r\omega_k)$.
We obtain $A_{k,n}\cong \bigoplus_{r\ge 0} V(r\omega_k)^*$ and
the Pl\"ucker coordinate $\bar p_{\uj}\in A_{k,n}$ for $\uj=(j_1<\dots<j_k)\in I_{k,n}$ is the dual basis element $(e_{j_1}\wedge\dots\wedge e_{j_k})^*\in V(\omega_k)^*$. 

Let $U_k^-\subset U^-$ be the subgroup generated by elements $\exp(zf_{i,j})=\mathds 1 + zf_{i,j}$, where $\varepsilon_i-\varepsilon_j\in R_k^+$.
Then $U_k^-$ is open and dense in $SL_n/P_k\cong \Gr(k,\mathbb C^n)$.
Hence, we have Pl\"ucker coordinates on $U_k^-$:  $\bar p_{\uj}$ is the $k\times k$-minor on columns $[k]$ and rows $\uj$.
Moreover, we have an isomorphism of the fields of rational functions $\mathbb C(\Gr(k,\mathbb C^n))\cong \mathbb C(U_k^-)$.


\subsection{Valuations}\label{sec:pre valuations}
We recall basic notions on valuations and Newton-Okounkov polytopes as presented in \cite{KK12}.

Consider $A=\mathbb C[x_1,\dots,x_n]/I$ with $I$ a homogeneous ideal and kernel of $\pi:\mathbb C[x_1,\dots,x_n]\to A$. 
Denote $\bar x_i:=\pi(x_i)$.
The standard grading on the polynomial ring induces a positive grading on $A=\bigoplus_{i\ge 0}A_i$.
Let $d$ be the Krull-dimension of $A$. 
Additionally, fix a linear order $\prec$ on the additive abelian group $\mathbb Z^d$. 

\begin{definition}\label{def: valuation}
A map $\val: A\setminus\{0\}\to (\mathbb Z^d,\prec)$ is a \emph{valuation}, if it satisfies for all $f,g\in A\setminus\{0\}$ and $c\in \mathbb C^*$ 

(i) \  $\val(f+g)\succeq \min \{\val(f),\val(g)\}$,

(ii) $\val(fg)=\val(f)+\val(g)$ and  

(iii) $\val(cf)=\val(f)$.
\end{definition}

Let $\val:A\setminus\{0\}\to(\mathbb Z^d,\prec)$ be a valuation. 
The image $\{\val(f)\mid f\in A\setminus\{0\}\}\subset \mathbb Z^d$ forms an additive semi-group. 
We denote it by $S(A,\val)$ and refer to it as the \emph{value semi-group}. 
The \emph{rank} of the valuation is the rank of the sublattice generated by $S(A,\val)$ in $\mathbb Z^d$.
We are interested in valuations of \emph{full rank}, i.e. $\rank(\val)=d$.

One naturally defines a $\mathbb Z^d$-filtration on $A$ by $F_{\val \succeq a}:=\{f\in A\setminus\{0\}\vert \val(f)\succeq a\}\cup \{0\}$ (and similarly $F_{\val\succ a}$). The associated graded algebra is 
\begin{eqnarray}\label{eq:def ass graded}
\gr_\val(A):=\bigoplus_{a\in \mathbb Z^d}F_{\val\succeq a}/F_{\val \succ a}.
\end{eqnarray}
If the filtered components $F_{\val \succeq a}/F_{\val \succ a}$  are at most one-dimensional for all $a\in\mathbb Z^d$, we say $\val$ has \emph{one-dimensional leaves}.
By \cite[Theorem~2.3]{KM16} the full-rank assumption on $\val$ implies that $\val$ has one-dimensional leaves.
In particular, by \cite[Remark~4.13]{BG09} we then have an isomorphism $\mathbb C[S(A,\val)]\cong \gr_\val(A)$.
To define a $\mathbb Z_{\ge 0}$-filtration on $A$ induced by $\val$ we make use of the following lemma:

\begin{lemma*}(\cite[Lemma~3.2]{Cal02})
Let $S$ be a finite subset of $\mathbb Z^d$. Then there exists a linear form $e:\mathbb Z^d\to \mathbb Z_{\ge 0}$ such that for all $\bm,\bn\in S$ we have $\bm \prec \bn \Rightarrow e(\bm)>e(\bn)$ (note the switch!).
\end{lemma*}

Assume $S(A,\val)$ is finitely generated, more precisely assume $S(A,\val)$ is generated by $\val(\bar x_1),\dots,\val(\bar x_n)$.
In this case $\{\bar x_1,\dots,\bar x_n\}$ is called a \emph{Khovanskii basis}\footnote{This term was introduced in \cite{KM16} generalizing the notion of SAGBI basis} for $(A,\val)$.
Now choose a linear form as in the lemma for $S$ being the set of generators.
We construct a $\mathbb Z_{\ge 0}$-filtration on $A$ by $F_{\le m}:=\{f\in A\setminus\{0\}\mid e(\val(f))\le m\}\cup\{0\}$ for $m\in \mathbb Z_{\ge 0}$.
Define similarly $F_{<m}$.
The associated graded algebra satisfies
\begin{eqnarray}
\gr_\val(A)\cong \bigoplus_{m\ge0}F_{\le m}/F_{<m}.
\end{eqnarray}
For $f\in A\setminus\{0\}$ denote by $\overline f$ its image in the quotient $F_{\le  e(\val(f))}/F_{<e(\val(f))}$, hence $\overline f\in\gr_\val(A)$. 
We obtain a family of $\mathbb C$-algebras containing $A$ and $\gr_\val(A)$ as fibers (see e.g. \cite[Proposition~5.1]{An13}) that can be defined as follows:

\begin{definition}\label{def: Rees algebra}
The \emph{Rees algebra} associated with the valuation $\val$ and the filtration $\{F_{\le m}\}_m$ is the flat $\mathbb C[t]$-subalgebra of $A[t]$ defined as
\begin{eqnarray}\label{eq: def Rees}
R_{\val,e}:=\bigoplus_{m\ge 0} (F_{\le m})t^m.
\end{eqnarray}
It has the properties that $R_{\val,e}/tR_{\val,e}\cong \gr_\val(A)$ and $R_{\val,e}/(1-t)R_{\val,e}\cong A$. 
In particular, it defines a flat family over $\mathbb A^1$ (the coordinate on $\mathbb A^1$ given by $t$). The generic fiber is isomorphic to $\Proj(A)$ and the special fiber is the toric variety $\Proj(\gr_\val(A))$, where $\Proj$ is taken with respect to the $\mathbb Z_{\ge 0}$-grading on $A$.
\end{definition}

Introduced by Lazarsfeld-Musta\c{t}\u{a} \cite{LM09} and Kaveh-Khovanskii \cite{KK12} we recall the definition of Newton-Okounkov body. 

\begin{definition}\label{def: val sg NO}
Let $\val:A\setminus \{0\}\to (\mathbb Z^d,\prec)$ be a valuation of full rank. 
The \emph{Newton-Okounkov body} is
\begin{eqnarray}\label{eq: def NO body}
\Delta(A,\val):=\overline{\conv(\bigcup_{i> 0}\{\val(f)/i \mid 0\not =f\in A_i \})}
\end{eqnarray}
\end{definition}
Anderson showed in \cite{An13} that if $\gr_{\val}(A)$ is finitely generated, then $\Delta(A,\val)$ is a rational polytope. 
Moreover, it is the polytope associated to the normalization of the toric variety $\Proj(\gr_{\val}(A))$.

The aim of \cite{B-quasival} is to give a criterion for when a valuation induces toric degeneration. 
Using Anderson's result, this translates to giving a criterion for when the corresponding value semi-group is finitely generated. 
We state the criterion \cite[Lemma~3 and Theorem~1]{B-quasival} below and use it to prove our main theorem in \S\ref{sec:proof}.
It uses the notion of initial ideals, that are defined as follows.
For $u\in \mathbb Z^{n}_{\ge 0}$ let $x^u$ to denote the monomial $x_1^{u_1}\cdots x_n^{u_n}\in \mathbb C[x_1,\dots, x_n]$.

\begin{definition}\label{def: init}
 Let $f = \sum a_u x^u\in \mathbb C[x_1, \ldots, x_n]$ and fix an element 
$w \in \RR^{n}$ (resp. $M\in\mathbb Z^{d\times n}$). 
Then the \emph{initial form} of $f$ with respect to $w$ (resp. $M$) is
\begin{eqnarray}\label{eq: def init form}
\mathrm{in}_w(f):= \sum_{w\cdot u \text{ minimal}} a_u x^u, \ \
\text{  resp.  } \ \
\init_M(f):=\sum_{\begin{smallmatrix}
Mm =\min_{\prec}\{M{u}\mid a_{ u}\not=0\}
\end{smallmatrix}} a_{m} x^{m}.
\end{eqnarray}
Let $I\subset \mathbb C[x_1, \ldots, x_n]$ be an ideal. Then its \emph{initial ideal} with respect to $w\in \RR^{n}$ (resp. $M\in\mathbb Z^{d\times n}$) is defined as
\[
\mathrm{in}_w(I):=\langle \mathrm{in}_w(f) \mid f\in I \rangle, \ \ \text{  resp.  }\ \ \init_M(I):=\langle \init_M(f) \mid f\in I \rangle.
\]
\end{definition}

\begin{definition}\label{def: wt matrix from valuation}
Given a valuation $\val:A\setminus\{0\}\to (\mathbb Z^d,\prec)$ 
 define the \emph{weighting matrix of $\val$} by
\[
M_\val:=(\val(\bar x_1),\dots,\val(\bar x_n))\in \mathbb Z^{d\times n}.
\]
That is, the columns of $M_\val$ are given by the images $\val(\bar x_i)$ for $i\in[n]$.
\end{definition}

\begin{theorem*}(\cite[Theorem~1]{B-quasival})\label{thm: val and quasi val with wt matrix}
Let $\val:A\setminus\{0\}\to (\mathbb Z^d,\prec)$ be a full-rank valuation
with $M_\val\in\mathbb Z^{d\times n}$ the weighting matrix of $\val$.
Then 
\[
S(A,\val) \ \text{is generated by} \{\val(\bar x_i)\}_{i\in[n]} \quad \Leftrightarrow \quad \init_{M_{\val}}(I) \text{ is prime and } M_{\val} \text{ is of full rank.}
\]
In this case $\{\bar x_1,\dots,\bar x_n\}$ forms a Khovanskii basis for $(A,\val)$ and
$\Delta(A,\val)=\conv(\val(\bar x_i)\mid i\in[n])$.
\end{theorem*}
 
In particular, the theorem (resp. its proof) implies that if $\init_{M_\val}(I)$ is prime, we have $\gr_{\val}(A)\cong \mathbb C[x_1,\dots,x_n]/\init_{M_\val}(I)$.
To make a connection to tropical geometry, more precisely the tropical Grassmannian, we rely on the following Lemma.

\begin{lemma*}(\cite[Lemmata~2\&3]{B-quasival})\label{lem: wt matrix val in trop}
Let $\val:A\setminus\{0\}\to (\mathbb Z^d,\prec)$ be a full-rank valuation and $M_\val$ the associated weighting matrix. Then $\init_{M_\val}(I)$ is monomial-free.
Moreover, there exists ${w}\in \mathbb Z^{n}$ with $\init_{w}(I)=\init_{M_\val}(I)$.
\end{lemma*}

\subsection{Tropical Geometry}\label{sec:tropical}
We recall basic notions on tropical geometry, in particular on the tropical Grassmannian. 
For details we refer to \cite{M-S} and \cite{SS04}.
For $u\in\mathbb Z^n$ as before we denote $x^u=x_1^{u_1}\cdots x_n^{u_n}\in\mathbb C[x_1^{\pm 1},\dots, x_n^{\pm 1}]$.

\begin{definition}\label{def: trop fct}
Let $f=\sum a_{u}{x}^{ u}\in \mathbb C[x_1^{\pm1},\dots,x_n^{\pm1}]$. The \emph{tropicalization} of $f$ is the function $f^{\trop}:\mathbb R^{n}\to \mathbb R$ given by 
\[
f^{\trop}(w):=\min\{ w\cdot u \mid  u\in \mathbb Z^{n} \text{ and } a_{ u}\neq 0\}.
\]
\end{definition}

If $ w-v=(m,\dots,m)$, for some $v,w\in\mathbb R^{n}$ and $m\in \mathbb R$, we have that the minimum in $f^{\trop}({w})$ and $f^{\trop}({ v})$
is achieved for the same ${u}\in \mathbb Z^{n}$ with $ a_{u}\neq 0$.

\begin{definition}(\cite[Definitions~3.1.1/2]{M-S})\label{def:trop hypersurf}
Let $f=\sum a_{ u}{x}^{ u}\in \mathbb C[x_1^{\pm1},\dots,x_n^{\pm1}]$ and $V(f)$ the associated hypersurface in the algebraic torus $T^{n}=(\mathbb C^*)^{n}$. Then the \emph {tropical hypersurface} of $f$ is
\[
\trop(V(f)):=\left\{  w\in \mathbb R^{n} \left| 
\begin{matrix}
\text{the minimum in }f^{\trop}( w)\\
\text{is achieved at least twice}
\end{matrix}\right.\right\} .
\]
Let $I$ be an ideal in $\mathbb C[x_1^{\pm1},\dots,x_n^{\pm1}]$. The \textit{tropicalization} of the variety $V(I)\subset T^{n}$ is defined as
\[
\trop(V(I)):=\bigcap_{f\in I}\trop(V(f))\subset \mathbb R^{n}.
\]
\end{definition}
For a projective variety $V(I)\subset \mathbb P^{n-1}$ with $I$ a homogeneous ideal in $\mathbb C[x_1,\dots,x_n]$ we consider the ideal $\hat I:=I\mathbb C[x_1^{\pm1},\dots,x_n^{\pm1}]$.
Then $V(\hat I)=V(I)\cap T^n$. 
The tropicalization of a projective variety is defined as $\trop(V(I)):=\trop(V(\hat I))$.

Recall the notion of initial ideal from Definiton~\ref{def: init} and consider for $w\in \mathbb R^n$ the initial ideal $\init_w(I)$.
By \cite[Theorem 15.17]{Eis13},
there exists a flat family over $\mathbb{C}$ whose generic fiber over $t\neq 0$ is isomorphic
to $\mathbb C[x_1,\dots,x_n]/I$ and whose special fiber over $t=0$ is isomorphic to $\mathbb C[x_1,\dots,x_n]/\init_{w}(I)$. 
It is given by the following family of ideals
\begin{eqnarray}\label{eq: groebner family}
\tilde{I}_t:=\left\langle t^{-\min_{ u}\{w\cdot u\}} f(t^{w_1}x_1,\ldots,t^{w_n}x_n)\left\vert f=\sum a_{{u}}x^{{u}}\ \in I   \right.\right\rangle\subset \mathbb C[t,x_1^{\pm 1},\dots,x_n^{\pm 1}].
\end{eqnarray}
Let $I_s$ denote the ideal $\tilde I_t\vert_{t=s}$. For $s\not=0$ the isomorphism $\mathbb C[x_1,\dots,x_n]/I_s\cong \mathbb C[x_1,\dots,x_n]/I_1=\mathbb C[x_1,\dots,x_n]_i/I$ is given by a ring automorphism of $\mathbb C[x_1,\dots,x_n]$ sending $I_s$ to $I$.
If $\init_{w}(I)$ is \emph{toric}, i.e. a binomial prime ideal, then $V(\init_{w}(I))$ is a toric variety (see e.g. \cite[Lemma 2.4.14]{M-S}) and flat degeneration of $V(I)$ with family defined by $\tilde{I}_t$. 
To find toric initial ideals, it is reasonable to consider the tropicalization of $V(I)$ as due to the \emph{Fundamental Theorem of Tropical Geometry} \cite[Theorem~3.2.3]{M-S} we have
\[
\trop(V(I))=
\left\{
     w\in\mathbb R^{n} 
        \mid     \init_{ w}(I) \text{  is monomial-free} 
\right\}.
\]
Further, by the \emph{Structure Theorem} \cite[Theorem 3.3.5]{M-S} $\trop(V(I))$ is the support of a pure rational $d$-dimensional polyhedral fan, where $d$ is the Krull-dimension of $I$.
It contains a linear subspace, called \emph{lineality space} defined as
\[
\trop_0(V(I)):=\{w\in\trop(V(I)) 
        \mid     \init_{ w}(I)=I\}.
\]
If $I$ is homogeneous we have $\mathbb R(1,\dots,1)\subset \trop_0(V(I))$.
It is often convenient to consider the quotient of $\trop(V(I))$ by the lineality space.
We can choose a fan structure on $\trop(V(I))$ by considering it as a subfan of the Gr\"obner fan of $I$. 
In particular, if $v$ and $w$ lie in the relative interior of a cone $C$, denoted by $v,w\in C^\circ$, if and only if
\[
\init_w(I)=\init_v(I).
\]
We therefore adopt the notation $\init_C(I):=\init_w(I)$ for an arbitrary $w\in C^\circ$.
We say a cone $C$ is \emph{prime}, if $\init_C(I)$ is a prime ideal.

\begin{definition}
The \emph{tropical Grassmannian}, denoted $\trop(\Gr(k,\mathbb C^n))\subset \mathbb R^{\binom{n}{k}}$ is the tropical variety of the Pl\"ucker ideal $\mathcal I_{k,n}$.
By \cite[Corollary~3.1]{SS04} it is a $k(n-k)+1$-dimensional polyhedral fan whose maximal cones are all of this dimension.
\end{definition}

We mainly focus on the tropicalization of $\Gr(2,\mathbb C^n)$. 
By \cite[Corollary~4.4]{SS04} $\trop(\Gr(2,\mathbb C^n))$ has the very nice property that every initial ideal $\init_C(\mathcal I_{2,n})$ associated to a maximal cone $C\subset \trop(\Gr(2,\mathbb C^n))$ is toric. 
Further, Speyer and Sturmfels show the following:

\begin{theorem*}(\cite[Theorem~3.4]{SS04})\label{thm: SS space of trees is trop Gr} 
The quotient $\trop(\Gr(2,\mathbb C^n))/\trop_0(\Gr(2,\mathbb C^n))\subset \mathbb R^{\binom{n}{2}}/\mathbb R^{n}$ intersected with the unit sphere is, up to sign, the \emph{space of phylogenetic trees} \cite{BHV01}. 
\end{theorem*}

The theorem implies that every maximal prime cone $C$ can be associated with a \emph{labeled trivalent tree} with $n$ leaves. 
The set of all labeled trivalent trees with $n$ leaves is denoted by $\mathcal T_n$.
A \emph{trivalent tree} is a graph with internal vertices of valency three and no loops or cycles of any kind. 
Non-internal vertices are called \emph{leaves} and the word \emph{labeled} refers to labeling the leaves by $1,\dots, n$.
We call an edge \emph{internal}, if it connects two internal vertices.

We label the standard basis of $\mathbb R^{\binom{n}{2}}$ by ordered sequences  $(i<j)\in I_{2,n}$ corresponding to Pl\"ucker variables. 
The following definition shows how to obtain a point in the relative interior of a maximal cone in $\trop(\Gr(2,\mathbb C^n))$ from a labeled trivalent tree. It follows from \cite[Theorem~3.4]{SS04}.

\begin{definition}\label{def:treedeg}
Let $T$ be a labeled trivalent tree with $n$ leaves.  Then the $(i<j)$'th entry of the weight vector ${w}_T\in \trop(\Gr(2,\mathbb C^n))$ is 
\[
-\#\{\text{internal edges on path from leaf } i \text{ to leaf } j \text{ in }T\}.
\]
For notational convenience we set
$\init_T(\mathcal I_{2,n}):=\init_{{\bf w}_T}(\mathcal I_{2,n})$. The corresponding maximal cone in $\trop(\Gr(2,\mathbb C^n))$ is denoted $C_T$. 
\end{definition}

\begin{remark}\label{rem:families}
Combining the above, we have that every trivalent labeled tree induces a toric degeneration of $\Gr(2,\mathbb C^n)$ with flat family as given in \eqref{eq: groebner family}.
A main result of \cite{KM16} is that there is a full-rank valuation with finitely generated value semi-group associated to prime cones in the tropicalization of a projective variety.
This yields a flat family using Rees algebras as in Definition~\ref{def: Rees algebra}.
\end{remark}

The symmetric group $S_n$ acts on $\mathcal T_n$ by permuting the labels of the leaves of a tree.
We also have an $S_n$-action on Pl\"ucker coordinates given by
\[
\sigma(p_{ij})={\text{sgn}(\sigma)}p_{\sigma^{-1}(i)\sigma^{-1}(j)} \text{ for }\sigma\in S_n.
\]
If $\sigma^{-1}(i)>\sigma^{-1}(j)$, we set $p_{\sigma^{-1}(i)\sigma^{-1}(j)}:=-p_{\sigma^{-1}(j)\sigma^{-1}(i)}$.
The $S_n$-action induces a ring automorphism of $\mathbb C[p_{ij}]_{i<j}$ for every $\sigma\in S_n$.
It sends $\init_T(\mathcal I_{2,n})$ to $\init_{\sigma(T)}(\mathcal I_{2,n})$ for every trivalent labeled tree $T\in \mathcal T_n$.
Denote the equivalence class of $T$ with respect to this action by $\mathtt T$. 
It is uniquely determined by the underlying \emph{(unlabeled) trivalent tree} with $n$ leaves, see for example Figure~\ref{fig:gr(2,4)}. We denote the set of trivalent trees by $\mathcal T_n/S_n$.

\begin{figure}[ht]
\begin{center}
\begin{tikzpicture}[scale=.5]

\draw (0,0) -- (1,1) -- (0,2);
\draw (1,1) -- (3,1) -- (4,2);
\draw (3,1) -- (4,0);

\end{tikzpicture}
\end{center}\caption{A trivalent tree with $4$ leaves.}\label{fig:gr(2,4)}
\end{figure}

Consider a trivalent tree $\mathtt T\in\mathcal T_n/S_n$. If there are two non-internal edges connected to the same internal vertex $c$, then we say $\mathtt T$ has a \emph{cherry} at vertex $c$.

\begin{lemma}\label{lem:cherry}
Every trivalent tree with $n\ge4$ leaves has a cherry.
\end{lemma}

\begin{proof}
We use induction on $n$. For $n=4$ Figure~\ref{fig:gr(2,4)} displays the only trivalent tree in $\mathcal T_4/S_4$ and we see, it has two cherries. Now consider a trivalent tree $\mathtt T'\in\mathcal T_{n+1}/S_{n+1}$. We remove one edge connected to a leaf and obtain a tree $\mathtt T\in\mathcal T_n/S_n$. By induction, $\mathtt T$ has a cherry at some vertex $c$. Adding the removed edge back there are two possibilities: either we add it to an internal edge (creating a new internal vertex), then the cherry also exists in $\mathtt T'$. Or we add it at an edge with a leaf, hence create a new cherry. 
\end{proof}

\section{Birational sequences}\label{sec:pre_birat}

We recall the definition of birational sequences due to Fang, Fourier, and Littelmann in \cite{FFL15} and associated valuations. 
After proving the central Lemma~\ref{lem:birat} we define a new class of birational sequences called \emph{iterated} in Definition~\ref{def:itseq}.

From now on let $d:=k(n-k)$ be the dimension of $\Gr(k,\mathbb C^n)$.
Consider a positive root $\beta\in R^+$.
The \emph{root subgroup} corresponding to $\beta$ is given by $U_{-\beta}:=\{\exp(zf_{\beta})\mid z\in \mathbb C\}\subset U^-$, where $\exp(zf_{\beta})=\mathds 1 + zf_{\beta}$.

\begin{definition}(\cite[Definition 2]{FFL15})\label{def:birat}
Let $S=(\beta_1,\dots,\beta_{d})$ be a sequence of positive roots. Then $S$ is called a \emph{birational sequence} for $\Gr(k,\mathbb C^n)$ if the multiplication map
\[
\psi_S:=\text{mult}:U_{-\beta_1}\times\dots\times U_{-\beta_{d}}\to U^-
\]
has image birational to $U_k^-$.
\end{definition}

Notice that $U_{-\beta_1}\times\dots\times U_{-\beta_{d}}\cong \mathbb A^d$.

\begin{example}\label{exp: birat seq, PBW and string}
The following are two first (and motivating) examples of birational sequences:
\begin{enumerate}
    \item The multiplication map $\prod_{\beta \in R_k^+} U_{-\beta}\to U^-$ has image $U_k^-$ and hence, is birational.
    In particular, every sequence containing all roots in $R_k^+$ (in arbitrary order) is a birational sequence called \emph{PBW-sequence}, see \cite[Example~1 and page 131]{FFL15}. 
    We distinguish between PBW-sequences $S$ and $S'$ when the order of the roots in $S$ is different from the order of the roots in $S'$.
    \item For $w_0$ the longest element in $S_n$, let $w_k\in S_n/\langle s_i\mid i\not=k \rangle$ denote a coset representative of $w_0$.
    Fix a reduced decomposition $\w_k=s_{i_1}\dots s_{i_{d}}$ of $w_k$ and let $S=(\varepsilon_{i_1}-\varepsilon_{i_1+1},\dots,\varepsilon_{i_d}-\varepsilon_{i_d+1})$ be the corresponding sequence of simple roots. 
    Then $S$ is a birational sequence, it is referred to as \emph{the reduced decomposition case} in \cite[Example~2]{FFL15}. 
\end{enumerate}
\end{example}
\noindent
The second example shows that repetitions of positive roots may occur in birational sequences. 
Our aim is to shed some light on sequences that are \emph{neither} PBW \emph{nor} associated to reduced decompositions for Grassmannians.
The following lemma allows us to construct such sequences for $\Gr(k,\mathbb C^{n+1})$ from sequences for $\Gr(k,\mathbb C^n)$. 

\begin{lemma}\label{lem:birat}
Let $S=(\beta_{1},\dots,\beta_{d})$ be a birational sequence for $\Gr(k,\mathbb C^n)$ and chose $I=\{i_1,\dots,i_k\}\subset[n]$ with $\vert I\vert =k$. Then the following is a birational sequence for $\Gr(k,\mathbb C^{n+1})$:
\begin{eqnarray}\label{seq}
S':=(\varepsilon_{i_1}-\varepsilon_{n+1},\dots,\varepsilon_{i_k}-\varepsilon_{n+1},\beta_{1},\dots,\beta_{d}).
\end{eqnarray}
\end{lemma}

\begin{proof}
The elements of $\text{im}(\psi_{S'})$ are of the form 
\begin{eqnarray*}
\exp(y_{1}f_{{i_1,n+1}})\cdots \exp(y_{k}f_{{i_k,n+1}})\exp(z_1f_{\beta_1})\cdots \exp(z_{d}f_{\beta_{d}})
=\begin{pmatrix}
1 & 0  &\dots& 0\\
a_{2,1} & 1 &\dots& 0\\
\vdots & \ddots  &\ddots & 0\\
a_{n+1,1}&\dots &a_{n+1,n}&1\\
\end{pmatrix},
\end{eqnarray*}
for $y_1,\dots,y_k,z_1,\dots,z_{d}\in\mathbb C$.
They satisfy $a_{n+1,j}=y_{1}a_{i_1,j}+\dots + y_{k}a_{i_k,j}$ for $1\le i\le n$. 
Deleting the last row and column gives the image of $\psi_S$, i.e. for $1\le i<j\le n$ we have $a_{i,j}\in\mathbb C[z_1,\dots,z_d]$.
We need to show that $\text{im}(\psi_{S'})$ is birational to $U_k^-\subset SL_{n+1}$ given that $\text{im}(\psi_S)$ is birational to $U_k^-\subset SL_n$.
Denote the corresponding map for $S$ by $\phi_S:\text{im}(\psi_S)\dashrightarrow U_k^-\subset SL_n$.
Recall that elements of $U_k^-$ are block matrices of form 
\begin{eqnarray*}
(x_{i,j})_{i,j}=\left(\begin{matrix}
\mathds 1 & 0 \\
* & \mathds 1
\end{matrix}\right),
\end{eqnarray*}
where non-trivial $x_{i,j}$ occur for $1\le j\le k$ and $k+1\le i \le n$.
Let $\ui\in I_{k,n}$ be the orderd sequence with entries $i_1,\dots,i_k$.
We define a birational map $\phi_{S'}:\text{im}(\psi_{S'})\dashrightarrow U^-\subset SL_{n+1}$ by specifying for $1\le j\le k$
\[
y_j\mapsto\frac{(-1)^{k-j+1}\bar p_{\ui\setminus i_j\cup n+1}}{\bar p_{\ui}}\quad \text{and} \quad 
a_{r,s} \mapsto
 \phi_{S}(a_{r,s})=x_{r,s} \quad \text{if} \quad 1\le r,s\le n.
\]
Then $\phi_{S'}(a_{n+1,j})=\sum_{l=1}^k \phi_{S'}(y_{l})\phi_S(a_{i_l,j})$. 
Before continuing with the computation note that on $U_k^-\subset SL_{n+1}$ we have $\bar p_{(1<\dots<k)\setminus j\cup i_l}((x_{r,s})_{r,s})=x_{i_l,j}$ for $1\le l\le k$.
If $j\le k$, making use of the Pl\"ucker relations \eqref{eq:plucker rel} this yields the following:
\begin{eqnarray*}
\phi_{S'}(a_{n+1,j}) &=& \sum_{l=1}^k \phi_{S'}(y_{l})\phi_S(a_{i_l,j})
= \sum_{l=1}^k \frac{(-1)^{k-l+1}\bar p_{\ui\setminus i_l\cup n+1}}{\bar p_{\ui}}x_{i_l,j} \\
&=& \sum_{l=1}^k \frac{(-1)^{k-l+1}\bar p_{\ui\setminus i_l\cup n+1}}{\bar p_{\ui}}\bar p_{(1<\dots<k)\setminus j \cup i_l}
= \bar p_{(1<\dots<k)\setminus j \cup n+1} = x_{n+1,j}.
\end{eqnarray*}
For $l>k$ we obtain $\phi_{S'}(a_{n+1,l})\in \mathbb C(x_{i,j}\vert 1\le j\le k, k+1\le i\le n)$.
In particular, we have $\mathbb C(\text{im}(\phi_{S'}))\cong \mathbb C(U_k^-)\cong \mathbb C(\Gr(k,\mathbb C^{n+1}))$.
By little abuse of notation we denote the invertible maps induced by $S$ between $\mathbb C(\Gr(k,\mathbb C^n))$ and $\mathbb C(\mathbb A^d)\cong \mathbb C(z_1,\dots,z_d)$ by $\psi_S^*$ and $(\psi_S^*)^{-1}$, respectively.
Induced by the pullback of $\psi_{S'}$ we consider the map between the rings of rational functions (also denoted $\psi_{S'}^*$):
\begin{eqnarray*}
\psi_{S'}^*:\mathbb C(\Gr(k,\mathbb C^{n+1})) &\to& \mathbb C(\mathbb A^{d+k}) \cong \mathbb C(z_1,\dots,z_d,y_1,\dots,y_k) \\
\bar p_{\uj} &\mapsto& \left\{
\begin{matrix}
\psi_S^*(\bar p_{\uj}) & \text{if}& \uj\in I_{k,n},\\
\sum_{l=1}^k y_l \psi_S^*((-1)^{\ell(\uj,i_l)}\bar p_{\uj \setminus n+1 \cup i_l}) & \text{if} & n+1\in \uj.
\end{matrix}\right.
\end{eqnarray*}
A straightforward computation reveals that the following map is indeed inverse to $\psi_{S'}^*$ above
\begin{eqnarray*}
(\psi_{S'}^*)^{-1}:\mathbb C(\mathbb A^{d+k}) &\to& \mathbb C(\Gr(k,\mathbb C^{n+1})) \cong \mathbb C(\bar p_{\uj})_{\uj\in I_{k,n+1}}\\
z_i &\mapsto& (\psi_S^*)^{-1}(z_i) \in \mathbb C(\bar p_{\uj})_{\uj \in I_{k,n}}, \\
y_j &\mapsto& \frac{(-1)^{k-j+1}\bar p_{\ui \setminus i_j \cup n+1}}{\bar p_{\ui}},
\end{eqnarray*}
which completes the proof.
\end{proof}

\begin{definition}\label{def:itseq}
For $k<n$ consider a birational sequence for $\Gr(k,\mathbb C^{k+1})$. Now extend it as in (\ref{seq}) to a birational sequence for $\Gr(k,\mathbb C^{k+2})$. Repeat this process until the outcome is a birational sequence for $\Gr(k,\mathbb C^n)$. Birational sequences of this form are called \emph{iterated}.
\end{definition}

We explain how to obtain a valuation from a fixed birational sequence $S=(\beta_1,\dots,\beta_{d})$ for $\Gr(k,\mathbb C^n)$ as constructed in \cite[\S7]{FFL15}. 
Define the \emph{height function} $\height:R^+\to \mathbb Z_{\ge 0}$ by $\height(\varepsilon_i-\varepsilon_j):=j-i$. 
Then the \emph{height weighted function} $\Psi_S: \mathbb Z^{d}\to \mathbb Z$ is given by 
\[
\Psi_S(m_1,\dots,m_{d}):=\sum_{i=1}^{d}m_i\height(\beta_i).
\] 
Let $<_{\text{lex}}$ be the lexicographic order on $\mathbb Z^{d}$. 
The \emph{$\Psi_S$-weighted reverse lexicographic order} $\prec_{\Psi_S}$ on $\mathbb Z^d$ is defined by setting for ${m,n}\in \mathbb Z^{d}$
\begin{eqnarray}\label{eq: def psi wt order}
{m}\prec_{\Psi_S}{n} :\Leftrightarrow \Psi_S({ m})<\Psi_S({n}) \text{ or } \Psi_S({ m})=\Psi_S({ n}) \text{ and } { m}>_{\text{lex}}{n}.
\end{eqnarray}
When the sequence $S$ is clear from the context we denote $\Psi_S$ by $\Psi$.

\begin{definition}(\cite[\S7]{FFL15})\label{def:valseq}
Let $f=\sum a_{u}{y}^{u}$ with $u\in\mathbb Z^d_{\ge0}$ be a non-zero polynomial in $\mathbb C[y_1,\dots,y_{d}]$. 
The valuation $\val_S:\mathbb C[y_1,\dots,y_{d}]\setminus \{0\}\to(\mathbb Z_{\ge 0}^{d},\prec_{\Psi})$ associated to $S$ is defined as
\begin{eqnarray}\label{eq:valseq}
\val_S(f):=\min{}_{\prec_\Psi}\{{u}\in\mathbb Z_{\ge 0}^{d}\mid a_{u}\ne 0\}.
\end{eqnarray}
We extend $\val_S$ to a valuation on $\mathbb C(y_1,\dots,y_{d})\setminus\{0\}$ by $\val_S(\frac{f}{g}):=\val_S(f)-\val_S(g)$ for $f,g\in \mathbb C[y_1,\dots, y_d]\setminus \{0\}$.
\end{definition}

Valuations of form \eqref{eq:valseq} are called \emph{lowest term valuations}.
As $S$ is a birational sequence, for every $f\in \mathbb C(\Gr(k,\mathbb C^n))$ there exists a unique element $\psi_S^*(f)\in\mathbb C(y_1,\dots,y_{d})$. 
Hence, we have a valuation on $\mathbb C(\Gr(k,\mathbb C^n))\setminus\{0\}$ given by $\val_S(f):=\val_S(\psi_S^*(f))$. 
We restrict to obtain
\[
\val_S:A_{k,n}\setminus\{0\}\to (\mathbb Z^d_{\ge 0},\prec_\Psi).
\]
Denote by $S(A_{k,n},\val_S)$ the associated value semi-group and the associated graded algebra by $\gr_S(A_{k,n})$.
For the images of Pl\"ucker coordinates $\bar p_{\uj}\in A_{k,n}$ we choose as before the notation $\overline{p_{\uj}}\in\gr_S(A_{k,n})$.
The weighting matrix for $\val_S$ is denoted $M_S$.

\begin{theorem}\label{thm: seq and trop k}
Let $S$ be an iterated sequence for $\Gr(k,\mathbb C^n)$ and $\val_S:A_{k,n}\setminus\{0\}\to (\mathbb Z^d_{\ge 0},\prec_\Psi)$ the corresponding valuation. 
Then there exists a cone $C\subset \trop(\Gr(k,\mathbb C^n))$ such that 
\[
\init_{M_{S}}(\mathcal I_{k,n})=\init_C(\mathcal I_{k,n}).
\]
\end{theorem}
\begin{proof}
We need to show that $\val_S$ is a full rank valuation. 
Then by \cite[Corollary~3]{B-quasival} the initial ideal $\init_{M_\val}(\mathcal I_{k,n})$ is monomial-free.
Moreover, there exists $w\in \mathbb Z^{\binom{n}{k}}$ with $\init_{ w}(\mathcal I_{k,n})=\init_{M_\val}(\mathcal I_{k,n})$.
In particular, we have a cone $C\subset \trop(\Gr(k,\mathbb C^n))$ with ${ w}\in C^\circ $.

To prove that $\val_S$ has full rank, it suffices to show that $M_S$ has full rank.
As $S$ is an iterated sequence we pursue by induction on $n$.
If $n=k+1$ we may assume that $S=(\varepsilon_1-\varepsilon_{k+1},\dots,\varepsilon_k-\varepsilon_{k+1})$.
In this case, after choosing an apropriate order on Pl\"ucker coordinates, $M_S$ contains the identity matrix as a submatrix.
Now assume that the claim is true for $n-1$ and let 
$S=(\varepsilon_{i_1}-\varepsilon_{n},\dots,\varepsilon_{i_k}-\varepsilon_{n},\beta_{1},\dots,\beta_{d'})$, where $(\beta_1,\dots,\beta_{d'})$ is an iterated sequence for $\Gr(k,\mathbb C^{n-1})$.
We have to show that the rows of $M_S$ corresponding to $\varepsilon_{i_1}-\varepsilon_{n},\dots,\varepsilon_{i_k}-\varepsilon_{n}$ are linearly independent.
Consider $\uj^1,\dots, \uj^k\in I_{k,n}$ such that for all $l\le k$ we have $i_1,\dots,i_{l-1},n\in \uj^l$ and $i_l\not \in \uj^l$. 
Compute the columns of $M_S$ corresponding to $\val_S(\bar p_{\uj^1}),\dots, \val_S(\bar p_{\uj^k})$ using \cite[Proposition~2]{FFL15}.
We see that the submatrix with rows corresponding to $\varepsilon_{i_1}-\varepsilon_{n},\dots,\varepsilon_{i_k}-\varepsilon_{n}$ is a $k\times k$-identity matrix.
Hence, the columns of $\val_S(\bar p_{\uj^1}),\dots, \val_S(\bar p_{\uj^k})$ are linearly independent and so $M_S$ is of full rank by induction.
\end{proof}

\begin{corollary}\label{cor:full rank}
For every iterated sequence $S$ the weighting matrix $M_S$ is of full rank.
\end{corollary}

\section{Iterated sequences for \texorpdfstring{$\Gr(2,\mathbb C^n)$}{}}\label{sec:proof}
In this subsection we prove Theorem~\ref{thm:main birat} stated in the introduction. 
After proving Proposition~\ref{prop: init M_S= init C_S} we can apply \cite[Theorem~1]{B-quasival} to complete the proof.
We focus on iterated sequences for $\Gr(2,\mathbb C^n)$ and start by computing the images of Pl\"ucker coordinates under the associated valuations.

Let $S=(\beta_1,\dots,\beta_d)$ be a birational sequence for $\Gr(2,\mathbb C^n)$.
Let $U(\lie n^-_S)\subset U(\lie n^-)$ be the subalgebra generated by monomials of form $f_{\beta_1}^{m_1}\dots f_{\beta_d}^{m_{d}}$.
We consider the irreducible highest weight representation $V(\omega_2)=\bigwedge^2\mathbb C^n=U(\mathfrak n_S^-)\cdot(e_1\wedge e_2)$. 
There exists at least one monomial of form ${\bf f}^{m}_S=f_{\beta_1}^{m_1}\dots f_{\beta_{d}}^{m_{d}}$ with the property ${\bf f}_S^m\cdot(e_1\wedge e_2)=e_i\wedge e_j$ for all $1\le i,j\le n$. 
Then by \cite[Proposition~2]{FFL15} we have
\begin{eqnarray}\label{eq: val seq on plucker}
\val_S(\bar p_{ij})=\min{}_{\prec_{\Psi}}\{{ m}\in \mathbb Z^{d}_{\ge 0}\mid {\bf f}_S^m\cdot(e_1\wedge e_2)=e_i\wedge e_j\}.
\end{eqnarray}

\begin{example}\label{exp:seq}
Consider $\Gr(2,\mathbb C^4)$ and the iterated sequence $S=(\varepsilon_{1}-\varepsilon_{4},\varepsilon_{2}-\varepsilon_{4},\varepsilon_{1}-\varepsilon_{3},\varepsilon_{2}-\varepsilon_{3})$, respectively $S'=(\varepsilon_{3}-\varepsilon_{4},\varepsilon_{2}-\varepsilon_{4},\varepsilon_{1}-\varepsilon_{3},\varepsilon_{2}-\varepsilon_{3})$. 
They are birational by Lemma~\ref{lem:birat}, as $(\varepsilon_{1}-\varepsilon_{3},\varepsilon_{2}-\varepsilon_{3})$ is of PBW-type for $\Gr(2,\mathbb C^{3})$. 
We compute the valuation $\val_S$ on Pl\"ucker coordinates. 
There are two monomials sending $e_1\wedge e_2$ to $e_3\wedge e_4$, namely 
\[
{\bf f}_S^{(1,0,0,1)}\cdot(e_1\wedge e_2)={\bf f}_S^{(0,1,1,0)}\cdot(e_1\wedge e_2)=e_1\wedge e_4.
\]
We have $\Psi_S(1,0,0,1)=\Psi_S(0,1,1,0)=4$, but $(1,0,0,1)>_{lex}(0,1,1,0)$. Hence, $\val_S(\bar p_{34})=(1,0,0,1)$. 

For $\val_{S'}$ we compute ${\bf f}_S^{(1,0,0,1)}(e_1\wedge e_2)= {\bf f}_S^{(0,1,0,0)}(e_1\wedge e_2)= e_1\wedge e_4$
Again, we have $\Psi_{S'}(1,0,0,1)=\Psi_{S'}(0,1,0,0)=2$, but as $(1,0,0,1)>_{lex} (0,1,0,0)$ it follows $\val_{S'}(\bar p_{14})=(1,0,0,1)$.
In Table~\ref{tab:val(2,4)} you can find the images of all Pl\"ucker coordinates under $\val_S$ and $\val_{S'}$.
\end{example}

\begin{table}[ht]
\begin{center}
\begin{tabular}{ |c|c|c|}
\hline
Pl\"ucker & $\val_S$ & $\val_{S'}$\\
 \hline
$\bar p_{12}$ & $(0,0,0,0)$ & $(0,0,0,0)$\\
$\bar p_{13}$ & $(0,0,0,1)$ & $(0,0,0,1)$\\
$\bar p_{23}$ & $(0,0,1,0)$ & $(0,0,1,0)$\\
$\bar p_{14}$ & $(0,1,0,0)$ & $(1,0,0,1)$\\
$\bar p_{24}$ & $(1,0,0,0)$ & $(1,0,1,0)$\\
$\bar p_{34}$ & $(1,0,0,1)$ & $(0,1,1,0)$\\
\hline
\end{tabular}
\end{center}\caption{Images of Pl\"ucker coordinates under the valuations $\val_S,\val_{S'}$ associated to $S=(\varepsilon_{1}-\varepsilon_{4},\varepsilon_{2}-\varepsilon_{4},\varepsilon_{1}-\varepsilon_{3},\varepsilon_{2}-\varepsilon_{3})$, respectively $S'=(\varepsilon_{3}-\varepsilon_{4},\varepsilon_{2}-\varepsilon_{4},\varepsilon_{1}-\varepsilon_{3},\varepsilon_{2}-\varepsilon_{3})$ for $\Gr(2,\mathbb C^4)$.}\label{tab:val(2,4)}
\end{table}

\begin{figure}
\begin{center}
\begin{tikzpicture}[scale=.25]
\node at (-5,-0.5) {$T_3=$};

\draw (0,0) -- (0,-3);
\draw (0,0) -- (3,2);
\draw (0,0) -- (-3,2);

\node at (0,-3.6) {\tiny $3$};
\node at (3.5,2.25) {\tiny $2$};
\node at (-3.5,2.25) {\tiny $1$};
\end{tikzpicture}
\end{center}
\caption{labeled trivalent tree with three leaves.}\label{fig:3leaves}
\end{figure}

From now on we fix an iterated birational sequence $S=(\varepsilon_{i_n}-\varepsilon_{n},\varepsilon_{j_n}-\varepsilon_n,\dots,\varepsilon_{i_3}-\varepsilon_3,\varepsilon_{j_3}-\varepsilon_3)$ for $\Gr(2,\mathbb C^n)$, where $i_l,j_l\in[l-1]$ with $i_l\not=j_l$ for all $3\le l\le n$.
Then Algorithm~\ref{alg:tree assoc to seq} associates to $S$ a labeled trivalent tree $T_S$ with $n$ leaves.

\begin{algorithm}[h]
\SetAlgorithmName{Algorithm}{} 
\KwIn{\medskip {\bf Input:\ }  An iterated sequence $S=(\varepsilon_{i_n}-\varepsilon_{n},\varepsilon_{j_n}-\varepsilon_n,\dots,\varepsilon_{i_3}-\varepsilon_3,\varepsilon_{j_3}-\varepsilon_3)$, the trivalent tree $T_3$ as in Figure~\ref{fig:3leaves}.}
\BlankLine
{\bf Initialization:} Set $k=4$, $T^S_{3}:=T_3$.\\
\For{$k\le n$}{Construct a tree $T_k^S$ from $T^S_{k-1}$ by replacing the edge with leaf $i_k$ in $T_{k-1}^S$ by three edges forming a cherry with leaves labeled by $i_k$ and $k$.\\
    \If{$k=n$}
    {{\bf Output:\ } The tree $T^S_n$.}
    \Else{Replace $k$ by $k+1$, $T^S_{k-1}$ by $T^S_k$ and start over.}}
\BlankLine
{\bf Output:\ } The tree $T_S:=T^S_n$ and the sequence $\mathbb T_S:=(T_n^S,\dots,T_3^S)$ of trees.
\label{alg:tree assoc to seq}
\caption{Associating a trivalent tree $T_S$ to an iterated sequence $S$.}
\end{algorithm}

\begin{definition}\label{def: T_S and C_S}
To an iterated sequence $S$ we associate the trivalent tree $T_S$ and the sequence of trees $\mathbb T_S=(T^S_n,\dots,T^S_3)$ that are the output of Algorithm~\ref{alg:tree assoc to seq}. 
Denote by $C_S$ the maximal cone in $\trop(\Gr(2,\mathbb C^n))$ corresponding to the tree $T_S$ by \cite[Theorem~3.4]{SS04}.
\end{definition}

\begin{example}\label{exp:treeseq}
Consider $S=(\varepsilon_4-\varepsilon_6,\varepsilon_5-\varepsilon_6,\varepsilon_2-\varepsilon_5,\varepsilon_3-\varepsilon_5,\varepsilon_2-\varepsilon_4,\varepsilon_3-\varepsilon_4,\varepsilon_1-\varepsilon_3,\varepsilon_2-\varepsilon_3)$, an iterated sequence for $\Gr(2,\mathbb C^6)$ .
We construct the trees $\mathbb T_S=(T^S_{3},T_4^S,T_5^S,T^S_6)$ by Algorithm~\ref{alg:tree assoc to seq}. 
Figure~\ref{fig:treeseq} shows the obtained sequence of trees.

\begin{figure}[h]
\begin{center}
\begin{tikzpicture}[scale=.25]
\draw (0,0) -- (0,-2.5);
\draw (0,0) -- (2.5,2);
\draw (0,0) -- (-2.5,2);
\node at (0,-3.1) {\tiny $1$};
\node at (-3,2.5) {\tiny $2$};
\node at (3,2.5) {\tiny $3$};
\draw[->,thick] (3.5,0) -- (6,0);

\begin{scope}[xshift=10cm] 
\draw (-3,2)-- (-1.5,0) -- (1.5,0) -- (3,2);
\draw (-3,-2) -- (-1.5,0);
\draw (1.5,0) -- (3,-2);
\node at (3.5,-2.5) {\tiny $1$};
\node at (-3.5,2.5) {\tiny $2$};
\node at (3.5,2.5) {\tiny $3$};
\node at (-3.5,-2.5) {\tiny $4$};
\draw[->,thick] (5,0) -- (7.5,0);

\begin{scope}[xshift=12.5cm]
\draw (-3,2)-- (-1.5,0) -- (4.5,0) -- (6,2);
\draw (-3,-2) -- (-1.5,0);
\draw (4.5,0) -- (6,-2);
\draw (1.5,0) -- (1.5,-2);
\node at (6.5,-2.5) {\tiny $1$};
\node at (-3.5,2.5) {\tiny $2$};
\node at (6.5,2.5) {\tiny $3$};
\node at (1.5,-2.6) {\tiny $4$};
\node at (-3.5,-2.5) {\tiny $5$};
\draw[->,thick] (8,0) -- (10.5,0);

\begin{scope}[xshift=15cm]
\draw (-3,2)-- (-1.5,0) -- (4.5,0) -- (6,2);
\draw (-3,-2) -- (-1.5,0);
\draw (4.5,0) -- (6,-2);
\draw (1.5,0) -- (1.5,-2) -- (0,-4);
\draw (1.5,-2) -- (3,-4);
\node at (6.5,-2.5) {\tiny $1$};
\node at (-3.5,2.5) {\tiny $2$};
\node at (6.5,2.5) {\tiny $3$};
\node at (-.5,-4.5) {\tiny $4$};
\node at (-3.5,-2.5) {\tiny $5$};
\node at (3.5,-4.5) {\tiny $6$};
\end{scope}
\end{scope}
\end{scope}
\end{tikzpicture}
\end{center}
\caption{The sequence $\mathbb T_S$ for $S$ as in Example~\protect{\ref{exp:treeseq}}.}\label{fig:treeseq}
\end{figure}
\end{example}

Let $M_S:=M_{\val_S}$ be the weighting matrix associated to $\val_S$ as in Definition~\ref{def: wt matrix from valuation}.
We compute the initial ideal $\init_{M_S}(\mathcal I_{2,n})$ of the Pl\"ucker ideal $\mathcal I_{2,n}$ to apply \cite[Theorem~1]{B-quasival}. 

\begin{proposition}\label{prop: init M_S= init C_S}
For every iterated sequence $S$ we have $\init_{M_S}(\mathcal I_{2,n})=\init_{C_S}(\mathcal I_{2,n})$.
\end{proposition}
\begin{proof}
Let $\{e_{ij}\}_{(i<j)\in I_{2,n}}$ be the standard basis of $\mathbb R^{\binom{n}{2}}$. Adopting the notation for monomials in the polynomial ring $\mathbb C[p_{ij}]_{i<j}$ we have
\[
R_{i,j,k,l}=p_{ij}p_{kl}-p_{ik}p_{jl}+p_{il}p_{jk}=\p^{e_{ij}+e_{kl}}-\p^{e_{ik}+e_{jl}}+\p^{e_{il}+e_{jk}}. 
\]
In particular, $M_S(e_{ij}+e_{kl})=\val_S(\bar p_{ij})+\val_S(\bar p_{kl})=\val_S(\bar p_{ij}\bar p_{kl})$ and $\init_{M_S}(R_{i,j,k,l})$ is the sum of those monomials in $R_{i,j,k,l}$ for which the valuation $\val_S$ of the corresponding monomials in $A_{2,n}$ is minimal with respect to $\prec_{\Psi}$.

Recall that $\init_{C_S}(\mathcal I_{2,n})=\langle \init_{C_S}(R_{i,j,k,l})\mid 1\le i<j<k<l\le n \rangle$ by \cite[Proof of Theorem~3.4]{SS04}. 
This implies that it is enough to prove the following claim.

\smallskip\noindent
\underline{\emph{Claim:}} For every Pl\"ucker relation $R_{i,j,k,l}$ with $1\le i<j<k<l\le n$ we have $\init_{C_S}(R_{i,j,k,l})=\init_{M_S}(R_{i,j,k,l})$. 

\smallskip\noindent
\underline{\emph{Proof of claim:}} We proceed by induction. For $n=4$ let $S=(\varepsilon_i-\varepsilon_4,\varepsilon_j-\varepsilon_4,\varepsilon_{i_3}-\varepsilon_3,\varepsilon_{j_3}-\varepsilon_3)$, i.e. the tree $T_S$ has a cherry labeled by $i$ and $4$. 
Consider the Pl\"ucker relation $R_{i,j,k,4}=p_{ij}p_{k4}-p_{ik}p_{j4}+p_{i4}p_{jk}$ with $\{i,j,k\}=[3]$. Then
\[
\init_{C_S}(R_{i,j,k,4})=p_{ij}p_{k4}-p_{ik}p_{j4}.
\]
Let $S'=(\varepsilon_{i_3}-\varepsilon_3,\varepsilon_{j_3}-\varepsilon_3)$ be the sequence for $\Gr(2,\mathbb C^{3})$ and denote by $\hat p_{rs}$ with $(r<s)\in I_{2,3}$ the Pl\"ucker coordinates in $A_{2,3}$. 
For $m=(m_{d-2},\dots,m_1)\in\mathbb Z^{d-2}$ and $m_d,m_{d-1}\in\mathbb Z$ write $(m_d,m_{d-1},m):=(m_d,m_{d-1},m_{d-2},\dots,m_1)$
We compute
\[
\val_S(\bar p_{i4})=(0,1,\val_{S'}(\hat p_{ij})),\ \
\val_S(\bar p_{j4})=(1,0,\val_{S'}(\hat p_{ij})),\ \text{  and  } \ 
\val_S(\bar p_{k4})=(1,0,\val_{S'}(\hat p_{ik})).
\]
This implies $\val_S(\bar p_{i4}\bar p_{jk})\succ_{\Psi}\val_S(\bar p_{ij}\bar p_{k4})=\val_S(\bar p_{ik}\bar p_{j4})$, and hence
$\init_{M_S}(R_{i,j,k,4})=\init_{C_S}(R_{i,j,k,4}).$

Assume the claim is true for $n-1$ and let $S=(\varepsilon_{i_n}-\varepsilon_n,\varepsilon_{j_n}-\varepsilon_n,\dots,\varepsilon_{i_3}-\varepsilon_3,\varepsilon_{j_3}-\varepsilon_3)$ be an iterated sequence for $\Gr(2,\mathbb C^n)$.  
Then $S'=(\varepsilon_{i_{n-1}}-\varepsilon_{n-1},\varepsilon_{j_{n-1}}-\varepsilon_{n-1},\dots,\varepsilon_{i_3}-\varepsilon_3,\varepsilon_{j_3}-\varepsilon_3)$ is an iterated sequence for $\Gr(2,\mathbb C^{n-1})$.
Denote by $\hat p_{ij}$ with $(i<j)\in I_{2,n-1}$ the Pl\"ucker coordinates in $A_{2,n-1}$.
As $\val_S(\bar p_{ij})=(0,0,\val_{S'}(\hat p_{ij}))$ for $i,j<n$ we deduce $\init_{C_S}(R_{i,j,k,l})=\init_{M_S}(R_{i,j,k,l})$ with $i,j,k,l<n$ by induction.
Consider a Pl\"ucker relation of form $R_{i,j,k,n}$ and compute
\[
\val_S(\bar p_{ln})=\left\{\begin{matrix}
(1,0,\val_{S'}(\hat p_{ri_n})), & \text{ if } l\not=i_n,\\
(0,1,\val_{S'}(\hat p_{i_nj_n})),& \text{ if } l=i_n.
\end{matrix}\right.
\]
As $(i_n,n)$ is a cherry in $T_S$ we observe that the associated weight vector $w_{T_S}\in C_S^\circ\subset\trop(\Gr(2,\mathbb C^n))$ satisfies $(w_{T_S})_{ln}=(w_{T_S})_{li_n}=(w_{T_{S'}})_{li_n}-1$.
In particular, for $i,j,k\not=i_n$ we have by induction $\init_{M_S}(R_{i,j,k,n})=\init_{C_S}(R_{i,j,k,n})$.
The only relations left to consider are of form $R_{i_n,j,k,n}$  for $j,k\in[n-1]\setminus\{i_n\}$. For $M_S$ we compute by the above
\[
\val_S(\bar p_{i_nj}\bar p_{kn})=\val_S(\bar p_{i_nk}\bar p_{jn})\succ_{\Psi} \val_S(\bar p_{i_nn}\bar p_{jk}).
\]
Hence, $\init_{M_S}(R_{i_n,j,k,n})=p_{i_nj}p_{kn}-p_{i_nk}p_{jn}$. As $(i_n,n)$ labels a cherry in $T_S$ we also have $\init_{C_S}(R_{i_n,j,k,n})=\init_{M_S}(R_{i_n,j,k,n})$.
\end{proof}

\begin{theorem}\label{thm:main birat}
For every iterated sequence $S$ we have $\gr_S(A_{2,n})\cong \mathbb C[p_{ij}]_{i<j}/\init_{C_S}(\mathcal I_{2,n})$.
Conversely, for every maximal prime cone $C\subset \trop(\Gr(2,\mathbb C^n))$ there exists a birational sequence $S$, such that $\mathbb C[p_{ij}]_{i<j}/\init_C(\mathcal I_{2,n})\cong \gr_S(A_{2,n})$.
\end{theorem}
\begin{proof}
Let $S$ be an iterated sequence. By Corollary~\ref{cor:full rank} we can apply \cite[Theorem~1]{B-quasival} and get
\begin{eqnarray*}
\begin{matrix}
\gr_S(A_{2,n}) &\cong & 
  \mathbb C[p_{ij}]_{i<j}/\init_{M_S}(\mathcal I_{2,n}) 
    &\stackrel{\text{Proposition~\ref{prop: init M_S= init C_S}}}{=}& \mathbb C[p_{ij}]_{i<j}/\init_{C_S}(\mathcal I_{2,n}).
    \end{matrix}
\end{eqnarray*}

For the second we pursue by induction:
for $n=3$ there is only one maximal cone $C\subset\trop(\Gr(2,\mathbb C^3))$. The sequence $S=(\varepsilon_1-\varepsilon_3,\varepsilon_2-\varepsilon_3)$ satisfies $C_S=C$.
Assume the claim is true for $n-1$ and consider a maximal cone $C\subset\trop(\Gr(2,\mathbb C^n))$ with associated tree $T_C$.
Let $\mathtt T_C$ denote the underlying unlabeled tree.
Find a cherry in $\mathtt T_C$ and remove it. 
Denote by $\mathtt T_C'$ the obtained tree with $n-1$ leaves.
By induction there exist an iterated sequence $S'$ for $\Gr(2,\mathbb C^{n-1})$ with associated tree $T_{S'}$ of shape $\mathtt T_C'$.
Now add back the cherry and label the additional leaf by $n$.
The cherry is labeled by some $i<n$ and $n$.
Hence, $S:=(\varepsilon_i-\varepsilon_n,\varepsilon_j-\varepsilon_n,S')$ for arbitrary $j<n$ with $i\not =j$ is an iterated sequence with $T_S$ of shape $\mathtt T_C$ and therefore $\gr_S(A_{2,n})\cong \mathbb C[p_{ij}]_{i<j}/\init_C(\mathcal I_{2,n})$.
\end{proof}

For an iterated sequence $S$ for $\Gr(2,\mathbb C^n)$ denote by $\mathtt T_i^S$ the (unlabeled) trivalent tree underlying the labeled trivalent tree $T^S_i$ with $i$ leaves in the tree sequence $\mathbb T_S$.
Algorithm~\ref{alg:tree assoc to seq} provides a tool for comparing whether two iterated sequences induce isomorphic flat toric degenerations. 
Construct $\mathbb T_{S_1},\mathbb T_{S_2}$ for two such sequences $S_1,S_2$ and consider $\mathtt T^{S_1}_n$ and $\mathtt T^{S_2}_n$. If $\mathtt T^{S_1}_n$ and $\mathtt T^{S_2}_n$ coincide then 
\begin{eqnarray}\label{eq:cong init}
\init_{T_{n}^{S_1}}(\mathcal I_{2,n})\cong\init_{T_n^{S_2}}(\mathcal I_{2,n}).
\end{eqnarray}

Let $\mathcal H_d:=\conv(e\in\{0,1\}^d)$ be the hypercube in $\mathbb R_{\ge 0}^d$.

\begin{corollary}\label{cor:hypersimplex}
For every iterated sequence $S$ for $\Gr(2,\mathbb C^n)$ the Newton-Okounkov polytope has $\binom{n}{2}$ vertices of form $\val_S(\bar p_{ij})$ for $1\le i<j\le n$. Further, it has no interior lattice points and satisfies
\[
\Delta(A_{2,n},\val_S)\subset \mathcal H_d.
\]
Moreover, if $S_1,S_2$ are two different iterated sequences for $\Gr(2,\mathbb C^n)$ with
$\mathtt T_n^{S_1}=\mathtt T_n^{S_2}$, then
the Newton-Okounkov polytopes are unimodularly equivalent\footnote{Two polytopes $P,Q\subset \mathbb R^d$ are called \emph{unimodularly equivalent} if there exists matrix $M\in GL_d(\mathbb Z)$ and $w\in \mathbb Z^d$
$ Q=f_M(P)+w $,
where $f_M(x)=xM$ for $x\in \mathbb R^d$. It is denoted $Q\cong P$ and by \cite[\S2.1 and \S2.3]{CLS11} implies that the associated projective toric varieties are isomorphic.}.
\end{corollary}

\begin{proof}
By Theorem~\ref{thm:main birat} the value semi-group $S(A_{2,n},\val_S)$ is generated by $\{\val_S(\bar p_{ij})\}_{i<j}$.
This implies $\Delta(A_{2,n},\val_S)=\conv(\val(\bar p_{ij}))_{i<j}$.
Observe from the proof of Proposition~\ref{prop: init M_S= init C_S} that $\val_S(\bar p_{ij})\in\{0,1\}^{d}$.
Therefore, $\val_S(\bar p_{ij})\not\in \conv(\val_S(\bar p_{kl})\vert (k<l)\not=(i<j))$.
Hence, for all $1\le i<j\le n$ every $\val_S(\bar p_{ij})$ is indeed a vertex and $\Delta(A_{2,n},\val_S)\subset \mathcal H_d$.

Regarding the second statement, let $S_1,S_2$ be two iterated sequences.
Assume we have $\mathtt T^{S_1}_n=\mathtt T_n^{S_2}$ and
consider the lattice ideals associated to each $\init_{C_{S_i}}(\mathcal I_{2,n})$ (see e.g. \cite[Lemma~2]{BLMM}).
The corresponding lattice points $\{l^i_j\}_{j=1,\dots,\binom{n}{2}}$ define polytopes in their ambient lattice which are unimodularly equivalent to $\Delta(A_{2,n},\val_{S_i})$ (see \cite[proof of Theorem~4]{BLMM} for the precise construction).
Therefore, \eqref{eq:cong init} implies $\Delta(A_{2,n},\val_{S_1})\cong \Delta(A_{2,n},\val_{S_2})$.
\end{proof}

Given a maximal cone $C_{T}\in\trop(\Gr(2,\mathbb C^n))$ with associated labeled trivalent tree $T$ the rays of $C$ are in bijection with \emph{internal edges} of $T$, i.e. those not adjacent to any leaf.
Internal edges give rise to partitions of $[n]$ into two sets (each of cardinality $\ge 2$) corresponding to the labelings of the leaves on either side of the edge.
Given a partition $(A,B)$ of $[n]$ the ray is $E_{A,B}=-\sum_{i\in A}\sum_{j\in B} e_{ij}\in \mathbb R^{\binom{n}{2}}$, see \cite[page 396]{SS04}.

\begin{corollary}\label{cor:seq of cones}
Let $S$ be an iterated sequence $S$ for $\Gr(2,\mathbb C^n)$ and 
for $n\ge 4$ we consider the projection $\pi_{n}:\mathbb R^{\binom{n}{2}}\to\mathbb R^{\binom{n-1}{2}}$ onto all coordinates that do not involve $n$. 
Then $\mathbb T_S=(T_3,T_4^S,\dots,T_n^S)$ corresponds to a sequence of cones $(C_{T_3},C_{T_4^S},\dots,C_{T_n^S})$ with $C_{T_i^S}\in \trop(\Gr(2,\mathbb C^i))$ and $\pi_i(C_{T_i^S})=C_{T_{i-1}^S}$.
\end{corollary}
\begin{proof}
The first property follows from Theorem~\ref{thm:main birat}.
For the second notice that if $(A,B)$ is a partition of $[n]$ then $(A\cap [n-1],B\cap[n-1])$ is a partition of $[n-1]$.
Moreover, $\pi_n(E_{A,B})=E_{A\cap[n-1],B\cap[n-1]}\in \mathbb R^{\binom{n-1}{2}}$.
For every $4\le i\le n$ the trees $T_i^S$ and $T_{i-1}^S$ have all interior edges in common, except the one adjacent to the cherry added when constructing $T_{i}^S$ from $T_{i-1}^S$. 
Hence, the rays of $C_{T_i^S}$ are mapped to the rays of $C_{T_{i-1}^S}$ by $\pi_n$.
\end{proof}

\begin{remark}
In \cite[Proposition~8.16]{KM16} the authors describe an initial ideal $\init_M(I)$ for an ideal $I\subset \mathbb C[x_1\dots,x_n]$ with respect to a matrix $M\in \mathbb R^{n\times d}$ as $\init_{u_1}(\dots\init_{u_d}(I)\dots)$, where $u_1,\dots,u_d$ are the columns of $M$.
This leads to a sequence of initial ideals of $I$, namely $I,\init_{u_d}(I),\init_{u_{d-1}}(\init_{u_d}(I))$ and so on.
The sequence of cones $(C_{T_3},C_{T_4^S},\dots,C_{T_n^S})$ from Corollary~\ref{cor:seq of cones} also yields a sequence of initial ideals
$\init_{T_3}(\mathcal I_{2,3}),\init_{T_4^S}(\mathcal I_{2,4}),\dots,\init_{T_{n}^S}(\mathcal I_{2,n})$.
This sequence is different in nature from the first as we are taking initial ideals of different ideals.
\end{remark}

The following definition allows us to interpret iterated sequences for $\Gr(2,\mathbb C^n)$ in a combinatorial way in Corollary~\ref{cor:treegraph} below.

\begin{definition}
The \emph{tree graph} $\mathcal T$ is an infinite graph whose vertices at level $i\ge 3$ correspond to trivalent trees with $i$ leaves. There is an arrow $\mathtt T\to \mathtt T'$, if $\mathtt T$ has $i$ leaves, $\mathtt T'$ has $i+1$ leaves and $\mathtt T'$ can be obtained from $\mathtt T$ by attaching a new boundary edge to the middle of some edge of $\mathtt T$. There is a unique source $\mathtt T_3$ at level $3$ (see Figure~\ref{fig:treegraph}).
\end{definition}

\begin{corollary}\label{cor:treegraph}
Every iterated sequence $S$ for $\Gr(2,\mathbb C^n)$ corresponds to a path from $\mathtt T_3$ to $\mathtt T^S_n$ in the tree graph $\mathcal T$.
\end{corollary}

\begin{proof}
The underlying unlabeled trees in the sequence $\mathbb T_S=(T_3,T_4^S,\dots,T_n^S)$ associated to $S$ define the path $\mathtt T_3\to\mathtt T_4^S\to\dots\to\mathtt T_n^S$ in $\mathcal T$.
\end{proof}

\footnotesize
\newcommand{\etalchar}[1]{$^{#1}$}

\bigskip
\noindent
Lara Bossinger\\
Universidad Nacional Aut\'onoma de M\'exico\\
Instituto de Matem\'aticas Unidad Oaxaca\\
Le\'on 2, altos, Oaxaca de Ju\'arez, Centro Hist\'orico \\ 68000 Oaxaca, M\'exico\\
Email: \texttt{lara@im.unam.mx}

\end{document}

%% file: packages.tex
\usepackage{fancyhdr}
{}
\usepackage[utf8]{inputenc}
\usepackage[top=1.15in, bottom=1.15in, left=1.17in, right=1.17in]{geometry}
\usepackage{amsthm,dsfont, amsmath, amssymb, amscd, latexsym, multicol, verbatim, enumerate, graphicx,xy, color}
\usepackage{amstext,amsfonts,amssymb,amscd,amsbsy,amsmath}
\usepackage{authblk}
\usepackage[nottoc,numbib]{tocbibind}
\usepackage[titletoc]{appendix}
\usepackage{tikz,float}
\usetikzlibrary{patterns}
\usepackage[english]{babel}
\usepackage{cite}
\usepackage{stmaryrd}
\usepackage[colorlinks=true, pdfstartview=FitV, linkcolor=blue, citecolor=blue, urlcolor=blue, breaklinks=true]{hyperref}

\usepackage{helvet}         
\usepackage{courier}        
\usepackage{type1cm}  
\usepackage{multicol}        
\usepackage[bottom]{footmisc}
\usepackage{longtable}

\usepackage{hyperref}
\usepackage{makeidx}         
\usepackage{graphicx}        

\makeindex             

\usepackage{amstext,amsfonts,amssymb,amscd,amsbsy,amsmath}
\usepackage[ruled,vlined]{algorithm2e}
\usepackage{algpseudocode}
\usepackage{tikz-cd}	
\usepackage{tikz}
\usetikzlibrary{arrows}

\usetikzlibrary{decorations.markings}
\usepackage{youngtab}

\numberwithin{equation}{section}

\theoremstyle{definition}
\newtheorem{definition}{Definition}
\newtheorem{proposition}{Proposition}
\newtheorem{theorem}{Theorem}
\newtheorem{lemma}{Lemma}
\newtheorem{corollary}{Corollary}
\newtheorem{remark}{Remark}
\newtheorem{example}{Example}

\newtheorem*{corollary*}{Corollary}
\newtheorem*{lemma*}{Lemma}
\newtheorem*{theorem*}{Theorem}
\newtheorem*{proposition*}{Proposition}
\newtheorem*{problem*}{Problem}

\newcommand{\w}{\underline{w}}
\newcommand{\p}{\mathbf{p}}

\newcommand{\boundellipse}[3]
{(#1) ellipse (#2 and #3)
}

\newcommand{\bm}{\mathbf{m}}
\newcommand{\bn}{\mathbf{n}}

\newcommand{\RR}{\mathbb{R}}

\DeclareMathOperator{\trop}{trop}

\DeclareMathOperator{\Proj}{Proj}

\newcommand{\gr}{{\operatorname*{gr}}}
\newcommand{\lie}{\mathfrak}
\newcommand{\Gr}{{\operatorname*{Gr}}}

\newcommand{\ui}{\underline{i}}
\newcommand{\uj}{\underline{j}}

\newcommand{\val}{\mathfrak{v}}

\DeclareMathOperator{\init}{in}

\DeclareMathOperator{\height}{ht}
\DeclareMathOperator{\conv}{conv}

\DeclareMathOperator{\rank}{rank}

%% file: resubmit_birat_April21.bbl
\begin{thebibliography}{BLMM17}

\bibitem[AB04]{AB04}
Valery Alexeev and Michel Brion.
\newblock Toric degenerations of spherical varieties.
\newblock {\em Selecta Math. (N.S.)}, 10(4):453--478, 2004.

\bibitem[And13]{An13}
Dave Anderson.
\newblock Okounkov bodies and toric degenerations.
\newblock {\em Math. Ann.}, 356(3):1183--1202, 2013.

\bibitem[BFF{\etalchar{+}}18]{BFFHL}
Lara Bossinger, Xin Fang, Ghislain Fourier, Milena Hering, and Martina Lanini.
\newblock Toric degenerations of {${\rm Gr}(2, n)$} and {${\rm Gr}(3, 6)$} via
  plabic graphs.
\newblock {\em Ann. Comb.}, 22(3):491--512, 2018.

\bibitem[BG09]{BG09}
Winfried Bruns and Joseph Gubeladze.
\newblock {\em Polytopes, rings, and {$K$}-theory}.
\newblock Springer Monographs in Mathematics. Springer, Dordrecht, 2009.

\bibitem[BHV01]{BHV01}
Louis~J. Billera, Susan~P. Holmes, and Karen Vogtmann.
\newblock Geometry of the space of phylogenetic trees.
\newblock {\em Adv. in Appl. Math.}, 27(4):733--767, 2001.

\bibitem[BLMM17]{BLMM}
Lara Bossinger, Sara Lamboglia, Kalina Mincheva, and Fatemeh Mohammadi.
\newblock Computing toric degenerations of flag varieties.
\newblock In {\em Combinatorial algebraic geometry}, volume~80 of {\em Fields
  Inst. Commun.}, pages 247--281. Fields Inst. Res. Math. Sci., Toronto, ON,
  2017.

\bibitem[Bos18]{Thesis}
Lara Bossinger.
\newblock {\em Toric degenerations: a bridge between representation theory,
  tropical geometry and cluster algebras}.
\newblock PhD thesis, Universit{\"a}t zu K{\"o}ln, 2018.

\bibitem[Bos20]{B-quasival}
Lara Bossinger.
\newblock Full-rank valuations and toric initial ideals.
\newblock {\em Int. Math. Res. Not.}, (rnaa071), 2020.

\bibitem[Cal02]{Cal02}
Philippe Caldero.
\newblock Toric degenerations of {S}chubert varieties.
\newblock {\em Transform. Groups}, 7(1):51--60, 2002.

\bibitem[CLS11]{CLS11}
David~A. Cox, John~B. Little, and Henry~K. Schenck.
\newblock {\em Toric varieties}.
\newblock American Mathematical Soc., 2011.

\bibitem[Eis95]{Eis13}
David Eisenbud.
\newblock {\em Commutative algebra}, volume 150 of {\em Graduate Texts in
  Mathematics}.
\newblock Springer-Verlag, New York, 1995.
\newblock With a view toward algebraic geometry.

\bibitem[FFL11]{FFL11}
Evgeny Feigin, Ghislain Fourier, and Peter Littelmann.
\newblock {PBW} filtration and bases for irreducible modules in type {$A_n$}.
\newblock {\em Transform. Groups}, 16(1):71--89, 2011.

\bibitem[FFL17]{FFL15}
Xin Fang, Ghislain Fourier, and Peter Littelmann.
\newblock Essential bases and toric degenerations arising from birational
  sequences.
\newblock {\em Adv. Math.}, 312:107--149, 2017.

\bibitem[GL96]{GL96}
Nicolae Gonciulea and Venkatramani Lakshmibai.
\newblock Degenerations of flag and {S}chubert varieties to toric varieties.
\newblock {\em Transform. Groups}, 1(3):215--248, 1996.

\bibitem[KK12]{KK12}
Kiumars Kaveh and Askold~G. Khovanskii.
\newblock Newton-{O}kounkov bodies, semigroups of integral points, graded
  algebras and intersection theory.
\newblock {\em Ann. of Math. (2)}, 176(2):925--978, 2012.

\bibitem[KM19]{KM16}
Kiumars Kaveh and Christopher Manon.
\newblock Khovanskii bases, higher rank valuations, and tropical geometry.
\newblock {\em SIAM J. Appl. Algebra Geom.}, 3(2):292--336, 2019.

\bibitem[LB15]{LB15}
Venkatramani Lakshmibai and Justin Brown.
\newblock {\em The {G}rassmannian variety}, volume~42 of {\em Developments in
  Mathematics}.
\newblock Springer, New York, 2015.
\newblock Geometric and representation-theoretic aspects.

\bibitem[LM09]{LM09}
Robert Lazarsfeld and Mircea Musta\c{t}\u{a}.
\newblock Convex bodies associated to linear series.
\newblock {\em Ann. Sci. \'Ec. Norm. Sup\'er. (4)}, 42(5):783--835, 2009.

\bibitem[MS15]{M-S}
Diane Maclagan and Bernd Sturmfels.
\newblock {\em Introduction to tropical geometry}, volume 161 of {\em Graduate
  Studies in Mathematics}.
\newblock American Mathematical Society, Providence, RI, 2015.

\bibitem[RW19]{RW17}
K.~Rietsch and L.~Williams.
\newblock Newton--{O}kounkov bodies, cluster duality, and mirror symmetry for
  {G}rassmannians.
\newblock {\em Duke Math. J.}, 168(18):3437--3527, 2019.

\bibitem[SS04]{SS04}
David Speyer and Bernd Sturmfels.
\newblock The tropical {G}rassmannian.
\newblock {\em Adv. Geom.}, 4(3):389--411, 2004.

\end{thebibliography}
